\documentclass[reqno]{amsart}

\usepackage{amsmath, amssymb, amsthm, amsbsy, mathtools}
\usepackage{setspace} 
\usepackage{enumerate}
\usepackage{graphicx}
\usepackage{dsfont}
\usepackage{bbm}
\usepackage{fullpage}
\usepackage{floatrow}
\usepackage{array}
\usepackage[unicode,psdextra]{hyperref}
\usepackage{bookmark}
\usepackage{xcolor}
\usepackage{comment}

\hypersetup
{
colorlinks,
linkcolor={red!50!black},
citecolor={blue},
urlcolor={blue!80!black}
}

\numberwithin{equation}{section}
\allowdisplaybreaks
\newtheorem{theorem}{Theorem}[section]
\newtheorem*{theorem*}{Theorem}
\newtheorem{lemma}{Lemma}[section]
\newtheorem{definition}{Definition}[section]

\newtheorem{remark}{Remark}[section]

\newtheorem{proposition}{Proposition}[section]
\newtheorem*{proposition*}{Proposition}

\def\C{{\mathbb C}}
\def\D{{\mathbb D}}
\def\E{{\mathbb E}}
\def\N{{\mathbb N}}

\def\R{{\mathbb R}}

\def\J{{\mathcal J}}
\def\I{{\mathcal I}}

\def\O{{\mathcal O}}

\def\cC{{\mathrm C}}

\def\eps{\varepsilon}
\def\l{\ell}
\def\f{f}

\def\dis{\displaystyle}

\newcommand{\ztp}{(0,2 \pi)}

\newcommand{\Envtwo}{\E [ N_{n} \ztp ]}

\newcommand{\asntoinfty}{\quad\mbox{as } n\to\infty}
\newcommand{\asmtoinfty}{\quad\mbox{as } m\to\infty}

\newcommand{\zk}{\zeta_{k}}

\newcommand{\ul}{u_{\l}}

\newcommand{\fim}{\varphi_m(x)}
\newcommand{\fimp}{\varphi'_m(x)}
\newcommand{\as}{A_n^*}
\newcommand{\bs}{B_n^*}
\newcommand{\cs}{C_n^*}

\newcommand{\abs}[1]{\left\lvert #1 \right\rvert}

\makeatletter
\newcommand{\pushright}[1]{\ifmeasuring@#1\else\omit\hfill$\displaystyle#1$\fi\ignorespaces}
\newcommand{\pushleft}[1]{\ifmeasuring@#1\else\omit$\displaystyle#1$\hfill\fi\ignorespaces}
\makeatother

\begin{document}

\title[Real zeros of random trigonometric polynomials with periodic coefficients]{Real zeros of random  trigonometric polynomials \\ with $ \ell $-periodic coefficients}

\author{Ali Pirhadi}
\address{Department of Mathematics and Statistics, Georgia State University, Atlanta, GA 30303, USA}
%
\email{apirhadi@gsu.edu}







\begin{abstract}
The large degree asymptotics of the expected number of real zeros of a random trigonometric polynomial $$ T_n(x) = \sum_ {j=0} ^{n} a_j \cos (j x) + b_j \sin (j x), \quad x \in (0,2\pi), $$
with i.i.d.~real-valued standard Gaussian coefficients is known to be $ 2n / \sqrt{3} $. 
In this article, we consider quite a different and extreme setting on the set of the coefficients of $ T_n $.
We show that a random trigonometric polynomial of degree $ n $ with  $ \l $\noindent-periodic i.i.d.~Gaussian coefficients is expected to have significantly more real zeros compared to the classical case with i.i.d.~Gaussian coefficients.
More precisely, the expected number of real zeros of $ T_n $ is  proportional to $ n $ with a proportionality constant $  \cC_{\l,r} \in (\sqrt{2},2] $, which is explicitly represented by a double integral formula.
The case $ r=0 $ is marked as a special one since in such a case $ T_n $ asymptotically obtains the largest possible number of real zeros.
\end{abstract}

\makeatletter
\renewcommand\subsection{\@startsection{subsection}{2}%
	\z@{.5\linespacing\@plus.7\linespacing}{.5em}%
	{\normalfont\scshape}}
\makeatother
\maketitle



\section{Introduction} \label{P3Sec1}
The study of zeros and level crossings of random functions, which is an interesting topic in mathematics and statistics, has a long history of almost a century. 
Bloch and P\'olya \cite{BP1931} were the first to study the expected number of real zeros of a polynomial with random coefficients.
Later, Littlewood and Offord gave explicit bounds for  the number of real zeros of random (algebraic) polynomials with uniform, standard normal, and Bernoulli coefficients, see \cite{LO1938, LO1939, LO1943}.
Ever since, there has been continual interest in the study of asymptotic estimate of the number of real zeros of random algebraic/trigonometric polynomials and the way these roots are distributed in the complex plane. 
In the case of random polynomials, the pioneering work of Kac \cite{Kac1943} showed that the expected number of real zeros of polynomials with i.i.d. standard normal coefficients is asymptotically $ (2/\pi)  \log n$. More precisely,
\begin{equation}\label{real-expected-Kac}
\E [N_n(\R)] = \dfrac{2}{\pi} \log n+ o(\log n), \asntoinfty,
\end{equation}
where $N_n(S) $, in general, denotes the number of real roots of a random function (in this case a random polynomial) in a set $ S \subset \R $. 
The asymptotic relation \eqref{real-expected-Kac} was later studied and enhanced by other mathematicians, including Hammersley \cite{H1956}, Ibragimov and Maslova \cite{IM1971,IM1971-2}, Wilkins \cite{Wil1973}, Edelman and Kostlan \cite{EK1995}, and recently H.~Nguyen, O.~Nguyen and Vu \cite{NNV2016}.

A key step in investigating the expected real zeros of random functions (in general) with centered Gaussian entries has been the renowned Kac-Rice formula, first introduced in the 1940's by Kac \cite{Kac1943} and independently by Rice \cite{Ric1944}, while studying random polynomials. 
In this article, we stick to one out of many different, but more or less equivalent, versions of Kac-Rice's formula, which was proved by Lubinsky, Pritsker and Xie \cite{LPX2016}.
\begin{proposition*}[Kac-Rice Formula] 
\normalfont
Let $ [a, b] \subset \R ,$ and consider real-valued functions $ f_{j}(x) \in C^1 [a, b] $,  $ 0 \leqslant j \leqslant n $. Define the random function
$ F_n (x) = \sum _{j=0} ^{n} a_{j} f_{j}(x),  $
where the $a_j$ are i.i.d.~random variables with Gaussian distribution $ \mathcal{N}(0,\sigma^{2}) $ and the auxiliary functions
\begin{align*}
& A_n(x) := \sum _{j=0} ^{n} (f_{j} (x))^2,
& B_n(x) := \sum _{j=0} ^{n} f_j (x) f'_{j} (x) ,
& & \text{and}  &
& C_n(x) := \sum _{j=0} ^{n} (f'_{j} (x))^2.
\end{align*}
If $ A_n(x)>0 $ on $ [a,b] ,$ and there is $ M \in \N $ such that $ F'_{n}(x) $
has at most $ M $ zeros in $ (a, b) $ for all choices of
coefficients, then the expected number of real zeros of $ F_n(x) $ in the interval $ (a, b) $ is given by
\begin{equation*}  \label{Kac}
\E [ N_{ n }  ( a , b ) ] 
= \frac{1}{\pi} \displaystyle \int_{a}^{b}  \dfrac{ \sqrt{A_n(x)C_n(x)-B_n(x)^2} }{A_n(x)} \, dx. \tag{\(\ast\)}
\end{equation*}
\end{proposition*}

Among the class of random functions, of great importance are random trigonometric polynomials
\begin{equation*}
T_n(x) := \sum_ {j=0} ^{n} a_j \cos (j x) + b_j \sin (j x), \quad x \in (0,2\pi), \end{equation*}
and random cosine polynomials of the form
\begin{equation*}
V_n(x) := \sum_{j=0}^{n} a_j \cos(j x) , \quad x \in \ztp ,
\end{equation*}
where the coefficients are real and chosen at random. It is quite easy to show that these polynomials have $ 2n $ complex roots, counted with multiplicity. When it comes to the number of real zeros of these polynomials, the famed work of Dunnage \cite{Dun1966} showed that for a random cosine polynomial with i.i.d.~standard Gaussian coefficients, 
\begin{equation} \label{Dunnage}
\Envtwo =\frac{2n}{\sqrt{3}} + o(n) , \asntoinfty.
\end{equation}
Further results on the real zeros of random trigonometric polynomials and their asymptotic behavior can be found through the works of Das \cite{Das1968}, Qualls \cite{Qua1970}, Wilkins \cite{Wil1991}, Flasche \cite{Fla2017}, and extensively in the book of Farahmand \cite{Far1998} and the references therein. 

\subsection{\centering Random trigonometric polynomials with dependent coefficients}

There are two main approaches to make the coefficients of a given random trigonometric (cosine) polynomial dependent. One is through the non-trivial correlated coefficients, and the other is to identify suitable, often large enough, number of the coefficients.
Herein, we give some examples for each direction and discuss the resulted asymptotic that has been established in each model. Let us assume that $ (a_j)_{j\geqslant0} $ is a sequence of identically distributed random variables with standard Gaussian distribution, i.e., $ a_j \sim \mathcal{N}(0,1) $. 
Since $ \mu_i=\mu_j=0 $ and $ \sigma_i^2=\sigma_j^2=1 $, it is trivial that the correlation between the coefficients satisfies
\begin{equation*}
\rho_{ij} 
:= \mathrm{corr}(a_i,a_j)
= \dfrac{\mathrm{cov}(a_i,a_j)}{\sigma_i \, \sigma_j}
= \E[a_i a_j].
\end{equation*}
We can indeed correlate the coefficients through a function known as correlation function. Assume that $ \rho: \N \to \R $ is a suitable function and set $ \E[a_i a_j]=\rho(\abs{i-j}) $, $ i \ne j $. 
It is obvious that $ \rho \not\equiv 0 $ corresponds to the case of dependent Gaussian coefficients. The first works in this direction (non-trivial  correlation functions) are due to Sambandham \cite{Sam1978} and Renganathan and Sambandham \cite{RS1984}.
The authors showed that the asymptotic relation \eqref{Dunnage} still remains valid for random cosine polynomials whose coefficients are correlated through either the constant correlation function $ \rho \equiv r $ or the geometric one defined as $ \rho: k \mapsto r^k $, $ r \in (0,1) $.

Angst, Dalamo and Poly \cite{ADP2019} considered a normalized type of random trigonometric polynomials, also known as Qualls' ensemble, of the form
\begin{equation*}
f_n(x) :=\frac{1}{\sqrt{n}}\sum _{j=1} ^{n} a_{j} \cos (j x) + b_{j} \sin (j x) , \quad x \in \ztp ,
\end{equation*}
where the $ a_j $ and $ b_j $ are standard Gaussian with the same correlation function $ \rho $, and they satisfy the independence condition  $ \E[a_i b_j]=0 $.
They proved that under mild assumptions on the spectral function $ \psi_\rho $ associated with $ \rho $, \eqref{Dunnage} indeed holds. 
The first example in this direction that provides an asymptotic expected number of real zeros of $ f_n $ remarkably deviating from the universal asymptotic \eqref{Dunnage} is the work of Pautrel \cite{Pau2020}.
His work concentrates on the expected number of real zeros of $ f_n $ whose coefficients satisfy the correlation function $ \rho: k \mapsto \cos(k\alpha) $, $ \alpha \geqslant 0 $. He proves that for all $ \eps>0 $ and $l\in (\sqrt{2},2] $, there exist $ \alpha=\alpha(l)\geq 0 $, and infinitely many $ n \in \N $, such that
\begin{equation*}
	\abs{\frac{\Envtwo}{n}-l} \leqslant \eps.
\end{equation*}
A recent paper of Angst, Pautrel and Poly \cite{APP2021} has generalized the main result of \cite{ADP2019}.  
The authors have considered Qualls' ensemble where the coefficients are standard Gaussian, correlated through the same $ \rho $ with its associated spectral measure $ \mu_\rho $ whose density component (with
respect to the Lebesgue measure $ \lambda $) is $ \psi_\rho $. Under these assumptions, they have proved that 
\begin{equation*}
\lim_{n \to \infty} \dfrac{\Envtwo}{n}= \dfrac{\lambda\{\psi_\rho=0\}}{\pi\sqrt{2}}+\dfrac{2\pi-\lambda\{\psi_\rho=0\}}{\pi\sqrt{3}},
\end{equation*}
provided that $ \psi_\rho $ is $ C^1 $ with
H\"older derivative on an open set of full measure. As an outcome of this result, see \cite[Corollary~2.1]{APP2021}, it is also shown that as long as $ \psi_\rho >0 $ a.e., one  obtains
\begin{equation*}
	\liminf_{n \to \infty} \dfrac{\E[N_n(a,b)]}{n}\geqslant \dfrac{b-a}{\pi\sqrt{3}}. 
\end{equation*}
The above lower bound in fact supports the conjuncture made in \cite{Pir2021} by this article's author, which claims that $ 2n/\sqrt{3} $ is the least asymptotic number of real zeros one can expect for high degree random trigonometric polynomials with dependent Gaussian coefficients.

One may adopt quite a different approach to the class of dependent coefficients. One reasonable way to make the coefficients of a given random trigonometric (cosine) polynomial dependent is to identify certain number of coefficients while taking advantage of a broad range of choices of identification models. 
The first model we would like to mention is palindromic random cosine polynomials studied by Farahmand and Li \cite{FL2012}.
They assumed that the first half of the coefficients of $ V_n $ are i.i.d. standard Gaussian, and the rest are just their copies identified in a palindromic way, that is, $ a_{j} = a_{ n - j } $. 
Their result that has been corrected in \cite[Appendix,~pp.~101--107]{Pir-diss} reveals that a random cosine polynomial with palindromic Gaussian coefficients obtains approximately $ 36\% $ more expected real zeros than the exemplary one with i.i.d.~Gaussian coefficients.
One can identify \emph{blocks} of Gaussian coefficients (of the same length) instead and still study the number of real zeros of random trigonometric polynomials. 
Two cases of random cosine polynomials with identified blocks were considered in \cite{Pir2020}. 
The first model discussed (identified adjacent blocks) is another example of random cosine polynomials with weakly dependent coefficients meaning that even though the coefficients are not independent, 
 the expected number of real zeros coincides with the universal asymptotic \eqref{Dunnage}. 
On the other hand, the second model, which involves only two identical block of a significantly large size, provides an example of random cosine polynomials with strongly dependent coefficients, namely the expected number of real zeros deviates from the universal asymptotic \eqref{Dunnage}.
More recently, the author has investigated the mean number of real zeros of $ V_n $ whose Gaussian coefficients are grouped in blocks of a fixed length $ \l>1 $ and identified in a palindromic fashion, see \cite[Theorem~2.1]{Pir2021}. 
It turns out that in this model, we have
\begin{equation*}  
	\Envtwo
	=  \frac{2 n}{ \sqrt{3}} \ \mathrm{K}_{\l} + \O (n^{3/4}),  \asntoinfty ,
\end{equation*} 
where 
\begin{align*} 
	& \mathrm{K}_{\l} \notag
	: = \frac{1}{\pi ^2} \displaystyle \int_{ 0 }^{ \pi } \int_{ 0 }^{ \pi /2 }
	\sqrt{1+  \dfrac{3 (1-\ul^2(s)) }{  (1+  \ul(s) \cos( t )) ^2 }} \, ds \, dt ,
\end{align*} 
and 
$ \ul(s) := \sin (\l s)/(\l \sin(s)). $

In this paper, our main goal is to give another example that goes in the direction of the second approach to dependent coefficients while noting that in this case the asymptotic expected number of real zeros still exceeds $ 2n / \sqrt{3} $.
This article is organized as follows: in Section \ref{P3Sec2}, we introduce  the class of $ \l $-periodic coefficients and discuss the expected number of real zeros of high degree random trigonometric polynomials with $ \l $-periodic Gaussian coefficients. Section \ref{P3Sec3} is devoted to the proofs of the results stated in Section \ref{P3Sec2}.

\section{Random trigonometric polynomials with $ \l $-periodic coefficients} \label{P3Sec2}

Our motivation behind studying the number of real zeros of random trigonometric (cosine) polynomials with $ \l $\noindent-periodic coefficients is the classical work of Szeg\H{o} \cite[p.~260]{Rem1998} as stated below.  

\begin{theorem*}[Szeg\H{o}]
Let $ f(z)=\sum_{j=0}^{\infty} a_{j} z^j$ be a power series with only finitely many distinct coefficients. Then either $ \D $ (the open unit disk) is the domain of holomorphy of $ f $ or $ f $ can be extended to a rational function $ \hat{f}(z)=p(z)/(1-z^k) ,$ where $ p(z)\in \C[z] $ and $ k \in \N. $
\end{theorem*} 

Proving the above theorem requires showing that if  $ \D $ is not the domain of holomorphy of $ f $, then from some coefficient on all coefficients are periodic, i.e., there exist $ \lambda, \mu \in \N $ with $ \lambda < \mu $ such that
$ a_{\lambda+j} = a_{\mu+j}$, $ j \in \N\cup \{0\}. $
Therefore, defining $ P(z):= \sum_{j=0}^{\lambda-1} a_{j} z^j$ and $ Q(z):= \sum_{j=\lambda}^{\mu-1} a_{j} z^j $, $ z \in \D $,
we may write, while setting $ k=\mu-\lambda$,
\begin{equation*}
f(z) = P(z) + Q(z) + Q(z)z^{k} +Q(z)z^{2k}+\cdots = P(z) + \dfrac{Q(z)}{1-z^{k}} , \quad z \in \D.
\end{equation*}

To the contrary, it turns out that trigonometric (cosine) series with only finitely many distinct coefficients behaves quite differently.
\begin{proposition}
$ V_\infty(z) :=\sum_{j=0}^{\infty} a_{j} \cos(jz)$ with only finitely many distinct coefficients diverges in the entire complex plane.
\end{proposition}

\begin{proof}
If we set $ w=e^{iz} $, we see that
\begin{equation*}
V_\infty(z) =\sum_{j=0}^{\infty} a_{j} \cos(jz) 
=  \frac{1}{2} \bigg( \sum_{j=0}^{\infty} a_{j} w^{j} + \sum_{j=0}^{\infty} a_{j} w^{-j}  \bigg).
\end{equation*}
We note that $  \limsup_{j \to \infty} \abs{a_j}^{1/j} =1$ since we have only finitely many distinct coefficients. 
Therefore, $ \sum_{j=0}^{\infty} a_{j} w^{j} $ diverges in $ \C \setminus \overline{\D} $, so does $ \sum_{j=0}^{\infty} a_{j} w^{-j} $  in $ \D $. Equivalently, $ V_\infty $ diverges in $ \C \setminus \R $.
Note that $V_\infty$ with only finitely many distinct coefficients diverges in $ [0,2\pi) $ by the Divergence  Test. 
This could be shown by assuming to the contrary that $ \lim_{j \to \infty} a_{j} \cos(jx_0) = 0  $ for some $ x_0 \in [0,2\pi) $. Let  $ \eps \in (0,M/2) $, where $ M:= \min \{ \abs{a_j} :a_j\neq 0 \} $. 
Thus, there exists $ J \in \N $ such that $ \abs{a_{j} \cos(jx_0)}<\eps $ for all $ j \geqslant J $.
In particular, $ \abs{\cos(Jx_0)} <\eps/M<1/2 $, and
\begin{align*}
\abs{a_{2J} \cos(2Jx_0)}
& \geqslant M \abs{\cos(2Jx_0)} = M(1-2 \cos^2(Jx_0)) 
> M (1-2\eps^2 /M^2)>\eps,
\end{align*}
where the last inequality follows from the fact that $ 2\eps^2 +M\eps - M^2 < 0 $, for all $ \eps \in (0,M/2) $. Thus, $ V_\infty$ with only finitely many distinct coefficients diverges  in $ [0,2\pi) $, and by periodicity in $ \R $.
\end{proof}
The above proposition still holds if we replace $ V_\infty $ with $ T_\infty(z):=\sum_{j=0}^{\infty} a_{j} \cos(jz) + b_{j} \sin(jz)$.
However, we are still interested in the number of real zeros of  partial sums of these infinite series whenever the coefficients are $ \l $\noindent-periodic, which is indeed a special case of the coefficients being of finitely many.
To state our desired results, we require the following definition:

\begin{definition}
\normalfont
Fix $ \l \in \N $ and let $ n \in \N $. We say $a_{0}, a_{1}, \ldots , a_{n}$ are $ \l $-periodic if $ a_{i+\l}=a_i $ for all $ 0 \leqslant i  \leqslant n -\l$.
\end{definition}

For a fixed $ \l \in \N $ and any large enough $ n \in \N $, It is trivial that $ n= \l m -1 +r$ for some $ m \in \N $  and $ r \in \{0, 1, \ldots ,  \l -1 \} $.
In what follows, we first investigate the roots of polynomials $ P_n $,  $ T_n $, and $ V_n $ with $ \l $-periodic coefficients when $ r=0 $.
We note that in this case, $ r=0 $, these polynomial asymptotically obtain the largest possible number of zeros. 

\begin{lemma} \label{PeriodicPn}
Fix $ \l \in \N $, and set $ n= \l m-1, \ m \in \N $.
Let $ P_n(z) =\sum_{j=0}^{n} a_{j} z^j$, where $ a_{0}, a_{1}, \ldots , a_{n} $  are $ \l $-periodic, i.e., $ a_{i+\l}=a_i $ for all $ 0 \leqslant i  \leqslant n -\l$. 
Then
\begin{align*}
P_n(z)
& 	= \dfrac{z^{\l m}-1}{z^{\l}-1} \sum_{k=0}^{\l-1} a_{k} z^k .
\end{align*}
\end{lemma}
\begin{remark}
\normalfont
From the above lemma, it is clear that $ P_n $ has at least $ \l(m-1)=n-\l+1 $  zeros that are all unimodular.
We also note that if $ \l=1 $, then all zeros are the $ m $\noindent-th roots of unity other than 1.
\end{remark}

\begin{theorem} \label{PeriodicTn}
\subpdfbookmark{Theorem 2.1}{PeriodicTn}
Fix $ \l \in \N $, and set $ n= \l m-1$,  $ m \in \N \setminus\{1\} $.
Let $ T_n(x) = \sum _{j=0} ^{n} a_{j} \cos (j x) + b_{j} \sin (j x) $,  $ x \in \ztp $.
Assume $ \{a_i\}_{i=0}^{\l-1} \cup  \{b_i\}_{i=0}^{\l-1}  $ is a family of i.i.d. random variables with Gaussian distribution $\mathcal{N}(0, \sigma^2) $. 
We also assume that the coefficients  are $ \l $-periodic.
Then
\begin{equation*}
\E[N_n(0,2\pi)] = n+1-\l+\sqrt{n^2 +\dfrac{ \l^2 -1}{3} }.
\end{equation*}
\end{theorem}
\begin{remark}
\normalfont 
If $ \l=1 $, then all the zeros of $ T_n $ happen to be real. In other words, all the roots of the associated (complex-valued) algebraic polynomial are unimodular, which coincides with the remark on the previous lemma and Proposition~2.1 of \cite{Pau2020}.
\end{remark}

\begin{theorem} \label{PeriodicVn}
\currentpdfbookmark{Theorem 2.2}{PeriodicVn}
Fix $ \l \in \N $, and set $ n= \l m-1$,  $ m \in \N $, 
and let $ V_n(x) = \sum _{j=0} ^{n} a_{j} \cos (j x)  $,  $ x \in \ztp $.
Assume $ A_0   $, defined as above, is a family of i.i.d. random variables with Gaussian distribution $\mathcal{N}(0, \sigma^2) $. 
We further assume that the $ a_j $ are $ \l $-periodic.
Then
\begin{equation*}
\E[N_n(0,2\pi)] = 2n + \O (n^{2/3}), \asntoinfty,
\end{equation*}
where the implied constant (in big $ \O $ notation) depends only on $ \l $.
\end{theorem}

The case when $ r \ne 0 $, which is of a more delicate setting, is considered only for a random trigonometric polynomial $ T_n $ due to the similarity of the computations and their estimates. In the following theorem, we give a closed form of the mean number of real zeros of $ T_n $ when the coefficients are $ \l $-periodic and under the assumption that  $ r \ne 0 $.

\begin{theorem} \label{PeriodicTn-nonzero}
\currentpdfbookmark{Theorem 2.3}{PeriodicTn-nonzero}
Fix $ \l \in \N \setminus \{1\}$, and set $ n= \l m-1+r$,  $ m \in \N  $ and $ r \in \{ 1, \ldots , \l -1 \} $.
Let $ T_n(x) = \sum _{j=0} ^{n} a_{j} \cos (j x) + b_{j} \sin (j x) $,  $ x \in \ztp $.
Assume $ \{a_i\}_{i=0}^{\l-1} \cup  \{b_i\}_{i=0}^{\l-1}  $ is a family of i.i.d. random variables with Gaussian distribution $\mathcal{N}(0, \sigma^2) $. 
We suppose that the $ a_j $ and $ b_j $ are $ \l $-periodic and define
\begin{align*} 
& \cC_{\l,r} \notag
: = \frac{1}{\pi ^2} \displaystyle \int_{ 0 }^{ \pi } \int_{ 0 }^{ \pi }
\sqrt{1+  \dfrac{ r (\l-r) \sin^2(s) }{  [(\l-r) \sin^2(t)+  r \sin^2(s+t)] ^2 }} \, ds \, dt .
\end{align*} 
Then
\begin{equation*}
\E[N_n(0,2\pi)] = n \, \cC_{\l,r} + \O (n^{4/5}),  \asntoinfty ,
\end{equation*}
where the implied constant depends only on $ \l $.
\end{theorem}

\noindent\textbf{Remark 2.4.} Even though $ r $ is assumed  nonzero, setting $ r=0 $  (equivalently $ \cC_{\l,0}=1 $) implies that the average number of \emph{probabilistic} zeros of $ T_n $ is asymptotically $ n $. Then, considering  $ n $ \emph{deterministic} roots of $ T_n $ in one period, the asymptotic expected number of real zeros of $ T_n $  is $ 2n $, which is now in accordance with the result of Theorem \ref{PeriodicTn}.

The following lemma gives a lower bound and an upper bound for the $ \cC_{\l,r} $.

\begin{lemma} \label{boundedcKrl}
For $ 0<r<\l $, we have $  \cC_{\l,r} \in (\sqrt{2},2] $.
\end{lemma}

\section{Proofs} \label{P3Sec3}

\begin{proof}[Proof of Lemma \ref{PeriodicPn}]
With the assumptions of Lemma \ref{PeriodicPn}, one can write
\begin{align*}
P_n(z)
& =\sum_{j=0}^{n} a_{j} z^j
=\sum_{k=0}^{\l-1} a_{k} \sum_{j=0}^{m-1} z^{k+\l j} \\
& =\bigg( \sum_{j=0}^{m-1} z^{\l j} \bigg) \bigg( \sum_{k=0}^{\l-1} a_{k} z^k \bigg) 
= \dfrac{z^{\l m}-1}{z^{\l}-1} \sum_{k=0}^{\l-1} a_{k} z^k ,
\end{align*}
which gives us at least $ \l(m-1)=n-\l+1 $ deterministic roots with modulus one.
\end{proof}

\subpdfbookmark{Proof of Theorem 2.1}{PF 2.1}
\begin{proof}[Proof of Theorem \ref{PeriodicTn}]
It follows from \cite[1.341(1,3),~p.~29]{GR1980} that
\begin{align} \label{P3c1}
& \sum_{j=0}^{r-1}  \cos ( 2p j + q ) x = \frac{\sin(rp x) \cos ((r-1)p+q) x }{\sin(p x)},
\end{align}
and 
\begin{align} \label{P3c2}
& \sum_{j=0}^{r-1}  \sin ( 2p j + q ) x = \frac{\sin(rp x) \sin ((r-1)p+q) x }{\sin(p x)}.
\end{align}
For $ x \in (0,2\pi)  $, we apply \eqref{P3c1} and \eqref{P3c2}  and observe that
\begin{align*}
T_n(x) & = \sum _{j=0} ^{n} a_{j} \cos (j x) + b_{j} \sin (j x) \\
& = \sum_{k=0}^{\l-1} a_{k} \sum_{j=0}^{m-1}\cos(k+\l j)x
+ \sum_{k=0}^{\l-1} b_{k} \sum_{j=0}^{m-1}\sin(k+\l j)x \\
& = \dfrac{\sin (m\l x/2)}{\sin (\l x/2)} \sum_{k=0}^{\l-1} \big[a_{k} \cos (k+ (m-1)\l/2)x + b_{k} \sin (k+ (m-1)\l/2)x \big] \\
& =: \dfrac{\sin (m\l x/2)}{\sin (\l x/2)} \, T_n^* (x).
\end{align*}
We first estimate the expected number of real zeros of $ T_n^* $.
We observe that, for $ x \in (0,2\pi) $,
\begin{align*}
A_n^*(x) 
& =  \sum_{k=0}^{\l -1} \big[\cos^2 (k+ (m-1)\l/2)x + \sin^2 (k+ (m-1)\l/2)x\big] = \l >0,
\end{align*}
\begin{align*} 
B_n^*(x) 
& = - \sum_{k=0}^{\l -1} [k+ (m-1)\l/2] \sin (k+ (m-1)\l/2)x \, \cos (k+ (m-1)\l/2)x \notag \\
& \quad + \sum_{k=0}^{\l -1} [k+ (m-1)\l/2] \sin (k+ (m-1)\l/2)x \, \cos (k+ (m-1)\l/2)x
=0,
\end{align*}
and
\begin{align*} 
C_n^*(x) 
& =  \sum_{k=0}^{\l -1} \big[k+ (m-1)\l/2\big]^2 \sin^2 (k+ (m-1)\l/2)x \notag \\
& \quad + \sum_{k=0}^{\l -1} \big[k+ (m-1)\l/2\big]^2 \cos^2 (k+ (m-1)\l/2)x \notag\\
& = \sum_{k=0}^{\l -1} \bigg(k+ \dfrac{\l(m-1)}{2}\bigg)^2 
= \sum_{k=0}^{\l -1} \bigg( k^2+ \l(m-1) k + \dfrac{\l^2(m-1)^2}{4} \bigg) \notag\\
& = \dfrac{(\l-1)\l(2\l-1)}{6} + \dfrac{(\l-1)\l^2 (m-1)}{2} + \dfrac{\l^3(m-1)^2}{4} \notag\\
& = \dfrac{\l(3\l^2 m^2 - 6\l m+ \l^2+2)}{12} = \dfrac{\l \big[3(\l m -1)^2 +(\l^2-1)\big]}{12} \notag\\
& = \dfrac{\l \big[3n^2 +(\l^2-1)\big]}{12}.
\end{align*}
Thus, it follows from Kac-Rice's formula \eqref{Kac} and after simplification that
\begin{align*}
\E [ N_n^* (0,2\pi) ]
& = \frac{1}{\pi} \displaystyle \int_{0}^{2\pi} \dfrac{\sqrt{\as(x)\cs(x) - \bs(x)^2}}{\as(x)} \, dx
=  \sqrt{n^2 +\dfrac{ \l^2 -1}{3} },
\end{align*}
where $ N_n^* (0,2\pi) $ denotes the number of zeros of $ T_n^* $ in $ (0,2\pi) $.
Let us define 
\begin{equation} \label{phim}
\varphi_m(x):= \dfrac{\sin (m\l x/2)}{\sin(\l x/2)} .
\end{equation}
We know that
\begin{equation*}
Z(\varphi_m) \cap [0,2\pi/\l]= \{ 2j\pi/m\l : 1 \leqslant j \leqslant m-1\}.
\end{equation*}
Therefore, $ \varphi_m $ has $ \l(m-1) = n+1-\l$ zeros in $ [0,2\pi] $. Taking into account these $  n+1-\l $ deterministic zeros of $ \varphi_m $, we have 
\begin{align*}  
\E [ N_{n} \ztp ] 
& = n+1-\l + \E [ N_{n}^* \ztp ] = n+1-\l+\sqrt{n^2 +\dfrac{ \l^2 -1}{3} }. \qedhere
\end{align*}
\end{proof}

\currentpdfbookmark{Proof of Theorem 2.2}{PF 2.2}
\begin{proof}[Proof of Theorem \ref{PeriodicVn}]
Fix $ a \in (0,1/2), $ and define $ E = [0,\pi] \setminus F $, where $ F = [0, n^{-a}) \cup (\pi - n^{-a}, \pi ] $.
For $ x \in [0,\pi]  $, we apply \eqref{P3c1} and observe that
\begin{align*}
V_n(x) & = \sum _{j=0} ^{n} a_{j} \cos (j x) 
= \sum_{k=0}^{\l-1} a_{k} \sum_{j=0}^{m-1}\cos(k+\l j)x \\
& = \dfrac{\sin (m\l x/2)}{\sin (\l x/2)} \sum_{k=0}^{\l-1} a_{k} \cos (k+ (m-1)\l/2)x .
\end{align*}
If we set $ \l=1 $, it is easy to check that $ V_n $ has exactly $ 2n $ deterministic zeros in on period, so we may assume that $ \l \in \N \setminus \{1\} $. 
Noe, recalling \eqref{phim}, we can write $ V_n(x) = \fim V_n^* (x) $,
where 
\begin{equation*}
V_n^* (x) := \sum_{k=0}^{\l-1} a_{k} \cos (k+ (m-1)\l/2)x  .
\end{equation*}
To discuss the expected number of real zeros of $ V_n^* $ in $ E $, we compute $ \as $, $ \bs $ and $ \cs $.
First, for $ x \in E $, we apply \eqref{P3c1} and obtain
\begin{align} \label{P3c3}
A_n^*(x) 
& =  \sum_{k=0}^{\l -1} \cos^2 (k+ (m-1)\l/2)x 
= \dfrac{1}{2} \sum_{k=0}^{\l -1} \big[1+\cos (2k+ (m-1)\l)x \big]  \notag \\ 
& = \dfrac{1}{2} \bigg( \l + \frac{ \sin(\l x) \cos (\l m -1)x}{\sin(x)} \bigg) =  \dfrac{\l ( 1 +  u_{\l}( x) \cos(nx))}{2}  ,
\end{align}
where  $ \ul (x) :=\sin(\l x)/\l \sin(x) $. 
Note that Markov's inequality (see Theorem 15.1.4 of \cite{RS2002}) guarantees that $ \abs{u_{\l}(x)} < 1 $ on $  E $ implying that $ A_n^* >0 $ on $ E $.
Moreover,
\begin{align*} 
B_n^*(x) 
& = - \sum_{k=0}^{\l -1} [k+ (m-1)\l/2] \sin (k+ (m-1)\l/2)x \, \cos (k+ (m-1)\l/2)x \\
& = - \dfrac{1}{2} \sum_{k=0}^{\l -1} [k+ (m-1)\l/2] \sin (2k+(m-1)\l)x \\ 
& = - \dfrac{\l m}{4} \sum_{k=0}^{\l -1} \sin (2k+(m-1)\l)x - \dfrac{1}{2} \sum_{k=0}^{\l -1} [k-\l/2] \sin (2k+(m-1)\l)x\\
& = - \dfrac{\l m}{4} \sum_{k=0}^{\l -1} \sin (2k+(m-1)\l)x +\O  (1)\\
& = - \dfrac{\l m \sin(\l x) \sin (\l m -1)x}{4\sin(x)}  +\O  (1)
= - \dfrac{\l^2 m u_{\l}( x) \sin(nx)}{4} +\O  (1),
\end{align*}
where the last sum is obtained with help of \eqref{P3c2}.
Therefore,
\begin{equation} \label{P3c4}
B_n^*(x) 
= - \dfrac{\l  n u_{\l}( x) \sin(nx)}{4}  +\O  (1), \asntoinfty \ \text{and} \ x \in E.
\end{equation}
In addition, with help of \eqref{P3c1}, we see 
\begin{align*} 
C_n^*(x) 
& =  \sum_{k=0}^{\l -1} [k+ (m-1)\l/2]^2 \sin^2 (k+ (m-1)\l/2)x \\
& 
= \dfrac{1}{2} \sum_{k=0}^{\l -1} [k+ (m-1)\l/2]^2 (1-\cos (2k+ (m-1)\l)x)   \notag \\ 
& = \dfrac{\l^2 m^2}{8}  \bigg(\l - \frac{ \sin(\l x) \cos (\l m -1)x}{\sin(x)}\bigg) 
+\O  (m) =  \dfrac{\l^3 m^2 (1 -u_{\l}(x) \cos(nx))}{8}  
+\O  (m).
\end{align*}
Hence
\begin{equation} \label{P3c5}
C_n^*(x)  = \frac{\l n^2 ( 1-   \ul(x) \cos(nx) ) }{8} 
+ \O  (n), \asntoinfty \ \text{and} \ x \in E.
\end{equation}
Thus, \eqref{P3c3}--\eqref{P3c5} imply that
\begin{align*} 
\Delta_n^*(x) := A_n^*(x) C_n^*(x) - B_n^*(x)^2  
& = \frac{\l^2 n^2 ( 1- \ul^2(x))}{16}  + \O (n), \asntoinfty \ \text{and} \ x \in E.
\end{align*}
Analogous to the proof of Lemma 3.4 of \cite{Pir2021}, we may show that, for $  x \in E $,
\begin{align*} 
\Delta_n^*(x) 
& = \frac{\l^2 n^2 ( 1- \ul^2(x) )\big[1+ \O (n^{-1+2a})\big]}{16}  , \asntoinfty.
\end{align*}
Therefore, 
\begin{align*} 
\dfrac{\sqrt{\Delta_n^*(x) } }{A_n^*(x)}
& = \frac{ n \big[1 + \O (n^{-1+2a})\big]}{2} \times \dfrac{\sqrt{1-\ul^2(x)}}{1+ \ul(x) \cos(nx)},  \asntoinfty \ \text{and} \ x \in E.
\end{align*}

From this point on, the proof is very close to that of Theorem 2.1 of \cite{Pir2021}. 
Set $ \dis G= E \cap [0,\pi/2]=[n^{-a},\pi/2] $.
Similar to the proof of Lemma 3.5 of \cite{Pir2021}, we can easily show that, as $ n $ tends to infinity,
\begin{align*} 
\E [ N^*_{n} \ztp ] & = 
\begin{cases}
\dfrac{\big[n+ \O (n^{2a})\big] \big[\J_{\l}^+(n)+\J_{\l}^-(n)\big]}{\pi } + \O (n^{1-a}), & \text{if $n-\l$ is even,} \\
& \\
\dfrac{\big[2n+ \O (n^{2a})\big] \J_{\l}^+(n)}{\pi }  + \O (n^{1-a}), & \text{if $n-\l$ is odd},
\end{cases}
\end{align*}
where 
\begin{align*}
\J_{\l}^{+} (n) := \int_{G} f_{n}^{+} (x) \, dx, 
\quad & \text{and} \quad  \J_{\l}^{-}(n) := \int_{G} f_{n}^{-} (x) \, dx,
\end{align*}
with
\begin{align*}
f_{n}^{+} (x)  
:=  \dfrac{\sqrt{1-\ul^2(x)}}{1+ \ul(x) \cos(nx)}, \quad
& \text{and}  \quad
f_{n}^{-} (x) := 
\dfrac{\sqrt{1-\ul^2(x)}}{1- \ul(x) \cos(nx)} . 
\end{align*}
Note that while proving Lemmas 3.6 and 3.7 of \cite{Pir2021}, we showed 
\begin{align*}
& \int_{0}^{n^{-a}} \f_{n}^+ (x) \, dx = \O (n^{-a}),
& & \text{and}  
& \int_{\pi/2n}^{n^{-a}}  \f_{n}^- (x) \, dx = \O (n^{-a}) .
\end{align*}
Likewise Lemma 3.8 of \cite{Pir2021} we can write
\begin{align} \label{P3c6}
\E [ N^*_{n} \ztp ] & = 
\begin{cases}
\big[n+ \O (n^{2a})\big] \big[\I_{\l}^+(n)+\I_{\l}^-(n)\big] + \O (n^{1-a}), & \text{if $n-\l$ is even,} \\
& \\
\big[2n+ \O (n^{2a})\big] \I_{\l}^+(n) + \O (n^{1-a}), & \text{if $n-\l$ is odd},
\end{cases}
\end{align}
where 
\begin{align*}
& \I_{\l}^{+}(n) := \dfrac{1}{\pi} \displaystyle \int_{0}^{\pi/2} f_{n}^{+} (x),
& & \text{and}  
& \I_{\l}^{-}(n) := \dfrac{1}{\pi} \displaystyle \int_{\pi/2n}^{\pi/2} f_{n}^{-} (x) \, dx .
\end{align*}
We use the same method established through Lemmas 3.9--3.11 of \cite{Pir2021} to show that
\begin{align*}
\I_{\l}^{-}(n) 
& =\dfrac{1}{\pi} \displaystyle \int_{\pi/2n}^{\pi/2} \dfrac{\sqrt{1-\ul^2(x)}}{1- \ul(x) \cos(nx)} \, dx \\
& =  \frac{1}{\pi ^2} \displaystyle \int_{ 0 }^{ \pi } \int_{ 0 }^{ \pi /2 }
\dfrac{\sqrt{1-\ul^2(s)}}{1- \ul(s) \cos(t)} \, ds \, dt +  \O (n^{-a}) \\
& = \frac{1}{\pi ^2} \displaystyle \int_{ 0 }^{ \pi/2 } \int_{ 0 }^{ \pi }
\dfrac{ \sqrt{1-\ul^2(s)} }{  1+  \ul(s) \cos( t ) } \, dt \, ds +  \O (n^{-a}) ,
\end{align*}
where the interchange of integration order is justified by the Fubini-Tonelli Theorem. Now, applying (2.1) of \cite{Pir2021} yields
\begin{equation*}
\int_{ 0 }^{ \pi }
\dfrac{dt}{1- \ul(s) \cos(t)}= \dfrac{\pi}{\sqrt{1-\ul^2(s)}},
\end{equation*}
which gives us
\begin{align*}
\I_{\l}^{-}(n) 
& =\frac 1 2 +  \O (n^{-a}) .
\end{align*}
Similarly, we have
\begin{align*}
\I_{\l}^{+}(n) 
& =\frac 1 2 +  \O (n^{-a}) .
\end{align*}
Plugging these two into \eqref{P3c6}, we may write
\begin{align*}
\E [ N^*_{n} \ztp ] & = n + \O (n^{2a}) + \O (n^{1-a}).
\end{align*}
It is clear that the best estimate in above occurs when $ a=1/3 $. Thus,
\begin{align*}
\E [ N^*_{n} \ztp ] & = n + \O (n^{2/3}).
\end{align*}
At last, considering $n+1-\l$ distinct (deterministic) roots of $ \varphi_m $, in one period, we have
\begin{align*}  
\E [ N_{n} \ztp ] 
& = n+1-\l + \E [ N_{n}^* \ztp ] = 2n + \O (n^{2/3}), \asntoinfty.
\end{align*}
\end{proof}
\currentpdfbookmark{Proof of Theorem 2.3}{here}
\label{here}
We prove Theorem \ref{PeriodicTn-nonzero} through a series of lemmas as follows. Keep in mind that in all of the following lemmas the implied constants   depend only on $ \l $.

\belowpdfbookmark{Lemma 3.1}{P2Lem3.1}
\begin{lemma}  \label{P2Lem3.1}
With the same assumptions as Theorem \ref{PeriodicTn-nonzero}, fix $ a \in (0,1/4) $ and let $E_0 = [0,\tfrac{2\pi}{\l}] \setminus F_0$, where $ F_0 = [0, \tfrac{2m^{-a}}{\l}) \cup (\tfrac{2\pi}{\l}- \tfrac{2m^{-a}}{\l},\tfrac{2\pi}{\l}] $.
Then
\begin{enumerate}[(1).] 
\item \label{P2Lem3.3-1}
$
0< A_n(x) = \l \fim^2 + r + 2r \fim \cos((m+1)\l x/2), \quad x \in E_{0}.
$ \\
\item \label{P2Lem3.3-2} 
\begin{multline*}
B_n(x)^2
 = \l^2 \fim^2 \fimp^2 + \dfrac{ r^2 m^2 \l^2 \cos^2((2m+1)\l x/2)}{4\sin^2(\l x/2)}  \\
  + \dfrac{ r m \l^2 \fim \fimp \cos((2m+1)\l x/2)}{\sin(\l x/2)} + \O (n^{1+4a}), \asntoinfty \ \text{and} \ x \in E_{0}.
\end{multline*}
\item \label{P2Lem3.3-3}
\begin{multline*}
C_n(x)
= \dfrac{m^2 \l^2 A_n(x)}{4}  + \dfrac{r m^2 \l^2 }{4} 
- rm \l \fimp \sin((m+1)\l x/2) \\
\quad + \l \fimp^2 + \O (n^{1+2a}), \asntoinfty \ \text{and} \ x \in E_{0}.
\end{multline*}
\end{enumerate}
where $ \varphi_m$ is as in \eqref{phim}, and $ A_n(x), \ B_n(x) $ and $ C_n(x) $ are defined in the same way as in Kac-Rice's formula \eqref{Kac}.
\end{lemma}
\begin{proof}
Since the coefficients are $ \l $-periodic, we can write
\begin{align*}
T_n(x)  & = \sum _{j=0} ^{n} a_{j} \cos (j x) + b_{j} \sin (j x) \\
& = \sum_{k=0}^{\l-1} a_{k} \sum_{j=0}^{m-1}\cos(k+\l j)x
+ \sum_{k=0}^{\l-1} b_{k} \sum_{j=0}^{m-1}\sin(k+\l j)x \\
 & \quad  + \sum_{k=0}^{r-1} a_{k} \cos(m\l+k)x
+ \sum_{k=0}^{r-1} b_{k} \sin(m\l+k)x.
\end{align*}
We apply \eqref{P3c1} and \eqref{P3c2} to simplify the above as
\begin{align*}
T_n(x) & =  \sum_{k=0}^{\l-1} a_{k} \fim \cos(k+(m-1)\l/2)x
+ \sum_{k=0}^{\l-1} b_{k} \fim \sin(k+(m-1)\l/2)x \\
& \quad + \sum_{k=0}^{r-1} a_{k} \cos(m\l+k)x
+ \sum_{k=0}^{r-1} b_{k} \sin(m\l+k)x \\
& =  \sum_{k=0}^{r-1} a_{k} [\fim \cos(k+(m-1)\l/2)x + \cos(m\l + k)x] \\
& \quad+ \sum_{k=0}^{r-1} b_{k} [\fim \sin(k+(m-1)\l/2)x + \sin(m\l + k)x] \\
& \quad+ \sum_{k=r}^{\l-1} a_{k} \fim \cos(k+(m-1)\l/2)x
+ \sum_{k=r}^{\l-1} b_{k} \fim \sin(k+(m-1)\l/2)x.
\end{align*}

\noindent \emph{{Proof of (1)}}. For $ x \in E_0 $, we have
\begin{align*}
A_n(x) & =  \sum_{k=0}^{r-1}  [\fim \cos(k+(m-1)\l/2)x + \cos(m\l + k)x]^2 \\
& \quad+ \sum_{k=0}^{r-1} [\fim \sin(k+(m-1)\l/2)x + \sin(m\l + k)x]^2 \\
& \quad+ \sum_{k=r}^{\l-1}  \fim^2 \cos^2(k+(m-1)\l/2)x
+ \sum_{k=r}^{\l-1} \fim^2 \sin^2(k+(m-1)\l/2)x \\
& = \l \fim^2 + r +2 \fim \sum_{k=0}^{r-1} \cos(k+(m-1)\l/2)x \cos(m\l + k)x \\
& \quad +2 \fim \sum_{k=0}^{r-1} \sin(k+(m-1)\l/2)x \sin(m\l + k)x .
\end{align*}
Thus,
\begin{align} \label{P3c7}
	A_n(x) & =  \l \fim^2 + r + 2r \fim \cos((m+1)\l x/2), \quad x \in E_{0}.
\end{align}
We also note that since $ 0<r<\l $, 
\begin{align*} 
A_n(x) & =  \l \bigg(\fim +\dfrac{r \cos((m+1)\l x/2)}{\l}\bigg)^2
+ r \bigg(1-\dfrac{r \cos^2((m+1)\l x/2)}{\l}\bigg)>0, \quad x \in E_{0}.
\end{align*}

\noindent \emph{{Proof of (2)}}. For $ x \in E_0 $, we observe that
\begin{align*}
& B_n(x) =  \sum_{k=0}^{r-1}  [\fim \cos(k+(m-1)\l/2)x + \cos(m\l + k)x] \\
& \times [\fimp \cos(k+(m-1)\l/2)x - \fim (k+(m-1)\l/2) \sin(k+(m-1)\l/2)x -
(m\l + k)\sin(m\l + k)x ]\\
& + \sum_{k=0}^{r-1}  [\fim \sin(k+(m-1)\l/2)x + \sin(m\l + k)x] \\
& \times [\fimp \sin(k+(m-1)\l/2)x + \fim (k+(m-1)\l/2) \cos(k+(m-1)\l/2)x +
(m\l + k)\cos(m\l + k)x ]\\
& + \sum_{k=r}^{\l-1}  \fim \cos(k+(m-1)\l/2)x \\
& \times [\fimp \cos(k+(m-1)\l/2)x - \fim (k+(m-1)\l/2) \sin(k+(m-1)\l/2)x]\\
& + \sum_{k=r}^{\l-1}  \fim \sin(k+(m-1)\l/2)x \\
& \times [\fimp \sin(k+(m-1)\l/2)x + \fim (k+(m-1)\l/2) \cos(k+(m-1)\l/2)x].
\end{align*}
After simplification, we obtain
\begin{align*}
B_n(x) & = \l \fim \fimp  + r \fimp \cos((m+1)\l x/2) \\
& \quad - \fim \sum_{k=0}^{r-1}  (m\l + k) \sin((m+1)\l x/2) 
+  \fim \sum_{k=0}^{r-1}  (k+(m-1)\l/2) \sin((m+1)\l x/2) \\
& = \l \fim \fimp + r \fimp \cos((m+1)\l x/2) - \dfrac{ r (m+1) \l \fim \sin((m+1)\l x/2)}{2}.
\end{align*}
We know that
\begin{align*}
\fimp =\dfrac{m\l \cos(m\l x/2)}{2\sin(\l x/2)} - \dfrac{\l \sin(m\l x/2) \cos(\l x/2)}{2\sin^2(\l x/2)}.
\end{align*}
Therefore, since $ \csc(\l x/2) = \O (n^{a})$ on $ E_{0} $, we may write 
\begin{align} \label{P3c8}
\fimp =\dfrac{m\l \cos(m\l x/2)}{2\sin(\l x/2)} +  \O (n^{2a}), \asntoinfty \ \text{and} \ x \in E_{0},
\end{align}
or if needed as in some occasions
\begin{align} \label{P3c9}
\fimp =m\O (n^{a}) +  \O (n^{2a})= \O (n^{1+a}) , \asntoinfty \ \text{and} \ x \in E_{0}.
\end{align}
Thus, we may rewrite $ B_n $ as
\begin{align} \label{P3c10}
B_n(x) & =  \l \fim \fimp + \dfrac{r m\l \cos(m\l x/2) \cos((m+1)\l x/2)}{2\sin(\l x/2)} +  \O (n^{2a}) \notag \\
& \quad - \dfrac{ r m \l  \sin(m\l x/2) \sin((m+1)\l x/2)}{2\sin(\l x/2)} +  \O (n^{a}) \notag \\
& = \l \fim \fimp + \dfrac{r m\l \cos((2m+1)\l x/2)}{2\sin(\l x/2)} +  \O (n^{2a})
, \asntoinfty \ \text{and} \ x \in E_{0}.
\end{align}
It follows from \eqref{P3c8}--\eqref{P3c10} and after simplification that
\begin{multline} \label{P3c11}
B_n(x)^2
 = \l^2 \fim^2 \fimp^2 + \dfrac{ r^2 m^2 \l^2 \cos^2((2m+1)\l x/2)}{4\sin^2(\l x/2)}  \\
+ \dfrac{ r m \l^2 \fim \fimp \cos((2m+1)\l x/2)}{\sin(\l x/2)} + \O (n^{1+4a}), \asntoinfty \ \text{and} \ x \in E_{0}.
\end{multline}
\noindent \emph{{Proof of (3)}}. For $ x \in E_0 $, we see that
\begin{align*}
C_n(x) & = \sum_{k=0}^{r-1}  [\fimp \cos(k+(m-1)\l/2)x - \fim (k+(m-1)\l/2) \sin(k+(m-1)\l/2)x  \\ 
& \quad- (m\l + k)\sin(m\l + k)x ]^2\\
& \quad+ \sum_{k=0}^{r-1}  [\fimp \sin(k+(m-1)\l/2)x + \fim (k+(m-1)\l/2) \cos(k+(m-1)\l/2)x \\
& \quad+ (m\l + k)\cos(m\l + k)x ]^2\\
& \quad+ \sum_{k=r}^{\l-1}  [\fimp \cos(k+(m-1)\l/2)x - \fim (k+(m-1)\l/2) \sin(k+(m-1)\l/2)x]^2\\
& \quad+ \sum_{k=r}^{\l-1}  [\fimp \sin(k+(m-1)\l/2)x + \fim (k+(m-1)\l/2) \cos(k+(m-1)\l/2)x]^2.
\end{align*}
Therefore, after simplification, we have
\begin{align*}
C_n(x) & = \l \fimp^2 + \fim^2 \sum_{k=0}^{\l-1} (k+(m-1)\l/2)^2 + \sum_{k=0}^{r-1} (k+m\l)^2 \\
& \quad -2 \fimp \sum_{k=0}^{r-1}  (m\l + k) \sin((m+1)\l x/2) \\
& \quad + 2 \fim \sum_{k=0}^{r-1} (m\l + k) (k+(m-1)\l/2) \cos((m+1)\l x/2) .
\end{align*}
It follows from \eqref{P3c9} along with the estimate $ \fim = \O (n^{a}) $ that, for $ x \in E_0 $, 
\begin{align*}
C_n(x) & = \l \fimp^2 + \dfrac{m^2 \l^3 \fim^2}{4} + \O (n^{1+2a}) + r m^2\l^2 + \O (n)  \\
& \quad-2 rm\l \fimp \sin((m+1)\l x/2) + \O (n^{1+a}) \\
& \quad + r m^2\l^2 \fim  \cos((m+1)\l x/2) + \O (n^{1+a})\\
& = \l \fimp^2 + \dfrac{m^2 \l^3 \fim^2}{4} + r m^2\l^2  -2 rm\l \fimp \sin((m+1)\l x/2) \\
& \quad + r m^2\l^2 \fim  \cos((m+1)\l x/2)  + \O (n^{1+2a}) \\
& =  \dfrac{m^2\l^2 [\l \fim^2+ r + 2r\fim \cos((m+1)\l x/2) ]}{4} \\
& \quad +  \dfrac{3rm^2\l^2}{4} + \dfrac{rm^2\l^2 \fim \cos((m+1)\l x/2)}{2}\\
& \quad -2 rm\l \fimp \sin((m+1)\l x/2)  + \l \fimp^2
+ \O (n^{1+2a}) \\
& =  \dfrac{m^2\l^2 A_n(x)}{4} + \dfrac{3rm^2\l^2}{4} + \dfrac{rm^2\l^2 \fim \cos((m+1)\l x/2)}{2} \\
& \quad -2 rm\l \fimp \sin((m+1)\l x/2)  + \l \fimp^2 + \O (n^{1+2a}).
\end{align*}
It is also clear that
\begin{align*}
& \dfrac{3rm^2\l^2}{4} + \dfrac{rm^2\l^2 \fim \cos((m+1)\l x/2)}{2} - rm\l \fimp \sin((m+1)\l x/2) \\
& = \dfrac{3rm^2\l^2}{4} + \dfrac{rm^2\l^2 \sin(m\l x/2) \cos((m+1)\l x/2)}{2 \sin(\l x/2)} \\
& \quad - \dfrac{rm^2\l^2 \cos(m\l x/2) \sin((m+1)\l x/2)}{2 \sin(\l x/2)}  + \O (n^{1+a}) \\
& = \dfrac{3rm^2\l^2}{4} -\dfrac{rm^2\l^2}{2} + \O (n^{1+a}) = \dfrac{rm^2\l^2}{4} + \O (n^{1+a}).
\end{align*}
Thus, we may simplify $ C_n $ as 
\begin{multline} \label{P3c12}
C_n(x)
 = \dfrac{m^2 \l^2 A_n(x)}{4}  + \dfrac{r m^2 \l^2 }{4} 
- rm \l \fimp \sin((m+1)\l x/2)  \\
 + \l \fimp^2  + \O (n^{1+2a}), \asntoinfty \ \text{and} \ x \in E_{0},
\end{multline}
as required.
\end{proof}

\belowpdfbookmark{Lemma 3.2}{P2Lem3.2}
\begin{lemma} \label{P2Lem3.2}
With the same assumptions as Lemma \ref{P2Lem3.1}, we have
\begin{align*} 
& \E[N_n ([0,2\pi/\l])] 
= \frac{m}{\pi} \displaystyle \int_{\tfrac{\pi}{m}}^{\pi-\tfrac{\pi}{m}}
\sqrt{1+  \dfrac{ r (\l-r) \sin^2(x) }{  [(\l-r) \sin^2(mx) + r \sin^2(m+1)x]^2}} \, dx + \O (n^{4/5}), \asntoinfty .
\end{align*} 
\end{lemma}
\begin{proof}
It follows from \eqref{P3c7} and \eqref{P3c12} that, for $ x \in E_0 $,
\begin{align} \label{P3c13}
A_n(x)C_n(x)
& = \dfrac{m^2 \l^2 A_n(x)^2}{4}  + A_n(x) \bigg[\dfrac{r m^2 \l^2 }{4} 
- rm \l \fimp \sin((m+1)\l x/2) + \l \fimp^2\bigg]  + \O (n^{1+4a}).
\end{align}
We note that
\begin{align} \label{P3c14}
& A_n(x) \bigg[\dfrac{r m^2 \l^2 }{4} - rm \l \fimp \sin((m+1)\l x/2) + \l \fimp^2\bigg] \notag \\
& = \bigg[ \l \fim^2 + r + 2r \fim \cos((m+1)\l x/2) \bigg] \bigg[\dfrac{r m^2 \l^2 }{4} - rm \l \fimp \sin((m+1)\l x/2) + \l \fimp^2\bigg] \notag \\
& = \l^2 \fim^2 \fimp^2 + \bigg[ \dfrac{r m^2 \l^3 \fim^2 }{4} + r \l \fimp^2 \bigg] \notag \\
& \quad + \bigg[ \dfrac{r^2 m^2 \l^2 }{4} - r^2 m \l \fimp \sin((m+1)\l x/2) + 
\dfrac{r^2 m^2 \l^2 \fim \cos((m+1)\l x/2)}{2}  \bigg] \notag \\
& \quad + \big[ 2r \l \fim \fimp^2 \cos((m+1)\l x/2) - 
r m \l^2 \fim^2 \fimp \sin((m+1)\l x/2)  \big] \notag \\
& \quad - \big[ 2r^2 m \l \fim \fimp \sin((m+1)\l x/2) \cos((m+1)\l x/2) \big].
\end{align}
We use \eqref{P3c8} and whenever needed \eqref{P3c9} to simplify each $ [\, \cdot \,] $ in \eqref{P3c14} in order to obtain a simplified estimate for $ \eqref{P3c13} $. It is clear that
\begin{align} \label{P3c15}
\dfrac{r m^2 \l^3 \fim^2 }{4} + r \l \fimp^2 
& = \dfrac{r m^2 \l^3 \sin^2(m\l x/2) }{4 \sin^2(\l x/2)} + \dfrac{r m^2 \l^3 \cos^2(m\l x/2) }{4 \sin^2(\l x/2)} + \O (n^{1+3a}) \notag \\
& = \dfrac{r m^2 \l^3 }{4 \sin^2(\l x/2)} + \O (n^{1+3a}), \asntoinfty.
\end{align}
It is also easy to show that 
\begin{align} \label{P3c16}
& \dfrac{r^2 m^2 \l^2 }{4} - r^2 m \l \fimp \sin((m+1)\l x/2) + 
\dfrac{r^2 m^2 \l^2 \fim \cos((m+1)\l x/2)}{2}  \notag \\
& = \dfrac{r^2 m^2 \l^2 }{4} - \dfrac{r^2 m^2 \l^2 }{2} 
\bigg[ \dfrac{\cos(m\l x/2) \sin((m+1)\l x/2) - \sin(m\l x/2) \cos((m+1)\l x/2) }{ \sin(\l x/2)} \bigg] \notag \\
& \quad + \O (n^{1+2a})  \notag \\
& = \dfrac{r^2 m^2 \l^2 }{4} - \dfrac{r^2 m^2 \l^2 }{2} + \O (n^{1+2a})
= - \dfrac{r^2 m^2 \l^2 }{4} + \O (n^{1+2a}).
\end{align}
Moreover,
\begin{align} \label{P3c17}
& 2r \l \fim \fimp^2 \cos((m+1)\l x/2) - r m \l^2 \fim^2 \fimp \sin((m+1)\l x/2)  \notag \\
& = r \l \fim \fimp \bigg[ \dfrac{m \l \cos(m\l x/2) \cos((m+1)\l x/2) }{ \sin(\l x/2)} + \O (n^{2a}) 
- \dfrac{ m \l \sin(m\l x/2) \sin((m+1)\l x/2) }{ \sin(\l x/2)} \bigg] \notag \\
& = \dfrac{r m \l^2 \fim \fimp \cos((2m+1)\l x/2) }{ \sin(\l x/2)} + \O (n^{1+4a}).
\end{align}
At last,
\begin{align} \label{P3c18}
& 2r^2 m \l \fim \fimp \sin((m+1)\l x/2) \cos((m+1)\l x/2) \notag \\
& =  r^2 m^2 \l^2  \bigg[ \dfrac{\cos(m\l x/2) \sin((m+1)\l x/2)}{\sin(\l x/2)} \bigg] \bigg[ \dfrac{\sin(m\l x/2) \cos((m+1)\l x/2)}{\sin(\l x/2)} \bigg] + \O (n^{1+3a}) \notag \\
& =  r^2 m^2 \l^2  \bigg[ \dfrac{ \sin((2m+1)\l x/2) + \sin(\l x/2)}{2\sin(\l x/2)} \bigg] \bigg[ \dfrac{ \sin((2m+1)\l x/2) - \sin(\l x/2)}{2\sin(\l x/2)} \bigg] + \O (n^{1+3a})  \notag \\
& = \dfrac{r^2 m^2 \l^2}{4} \bigg[ \dfrac{ \sin^2((2m+1)\l x/2) }{\sin^2(\l x/2)} -1 \bigg] + \O (n^{1+3a}).
\end{align}
Putting together \eqref{P3c13}--\eqref{P3c18}, we have
\begin{align*} 
 A_n(x)C_n(x)
& = \dfrac{m^2 \l^2 A_n(x)^2}{4}  + \l^2 \fim^2 \fimp^2 + \dfrac{r m^2 \l^3 }{4 \sin^2(\l x/2)} \\
& \quad - \dfrac{r^2 m^2 \l^2 }{4} + \dfrac{r m \l^2 \fim \fimp \cos((2m+1)\l x/2) }{ \sin(\l x/2)} \\
& \quad + \dfrac{r^2 m^2 \l^2}{4} \bigg[ 1 - \dfrac{ \sin^2((2m+1)\l x/2) }{\sin^2(\l x/2)} \bigg] + \O (n^{1+4a}) \\
&  = \dfrac{m^2 \l^2 A_n(x)^2}{4}  + \l^2 \fim^2 \fimp^2 + \dfrac{r m^2 \l^3 }{4 \sin^2(\l x/2)} \\
& \quad + \dfrac{r m \l^2 \fim \fimp \cos((2m+1)\l x/2) }{ \sin(\l x/2)} \\
& \quad -\dfrac{r^2 m^2 \l^2 \sin^2((2m+1)\l x/2) }{4\sin^2(\l x/2)} + \O (n^{1+4a}).
\end{align*}
This together with \eqref{P3c11} implies that, for $ x \in E_0 $,
\begin{align*} 
\Delta_n(x) & := A_n(x) C_n(x) - B_n(x)^2 \\
& = \dfrac{m^2 \l^2 A_n(x)^2}{4} + \dfrac{r m^2 \l^3 }{4 \sin^2(\l x/2)} 
- \dfrac{r^2 m^2 \l^2 }{4 \sin^2(\l x/2)} + \O (n^{1+4a}) \\
& = \dfrac{m^2 \l^2 }{4} \bigg[ A_n(x)^2 + \dfrac{r(\l-r)}{\sin^2(\l x/2)} \bigg] +  \O (n^{1+4a}), \asntoinfty.
\end{align*}
We also note that since $ 0< r < \l $,
\begin{align*}
0 < \dfrac{1}{A_n(x)^2 + \dfrac{r(\l-r)}{\sin^2(\l x/2)}} \leqslant \dfrac{1}{r(\l-r)}<\infty,
\end{align*}
which implies that 
\begin{align*} 
\Delta_n(x) = \dfrac{m^2 \l^2 }{4} \bigg[ A_n(x)^2 + \dfrac{r(\l-r)}{\sin^2(\l x/2)} \bigg] \big[1 + \O (n^{-1+4a})\big].
\end{align*}
Since $ a \in (0,1/4) $ and $ A_n(x)>0 $, we shall write
\begin{align*} 
\dfrac{ \sqrt{\Delta_n(x)} }{A_n(x)} =  \frac{m\l (1+ \O (n^{-1+4a}))}{2} \
\sqrt{1+  \dfrac{ r (\l-r)  }{  \sin^2(\l x/2) A_n(x)^2 }}.
\end{align*}
We also see that
\begin{align*}
& \sin(\l x/2) A_n(x) = \sin(\l x/2) [\l \fim^2 + r + 2r \fim \cos((m+1)\l  x/2)] \\
& = \dfrac{\l \sin^2(m\l x/2) + r \sin^2(\l x/2) + 2r \sin(m\l x/2) \sin(\l x/2) \cos((m+1)\l  x/2)}{\sin(\l x/2)} \\
& = \dfrac{(\l-r) \sin^2(m\l x/2) + r [\sin^2(m\l x/2) + \sin^2(\l x/2) + 2 \sin(m\l x/2) \sin(\l x/2) \cos((m+1)\l  x/2)]}{\sin(\l x/2)}.
\end{align*}
For any $ a,b \in \R $, one can easily show that
\begin{align*}
& \sin^2(a) +\sin^2(b) + 2\sin(a)\sin(b)\cos(a+b) \\
& = \sin^2(a) +\sin^2(b) - 2\sin^2(a)\sin^2(b) + 2\sin(a)\sin(b)\cos(a)\cos(b) \\
& = \sin^2(a) \cos^2(b)+\sin^2(b) \cos^2(a) + 2\sin(a)\sin(b)\cos(a)\cos(b)
= \sin^2(a+b).
\end{align*}
Therefore, the above trigonometric identity helps us write
\begin{align*}
& \sin^2(\l x/2) A_n(x)^2 = \dfrac{[(\l-r) \sin^2(m\l x/2) + r \sin^2((m+1)\l  x/2)]^2}{\sin^2(\l x/2)},
\end{align*}
which gives us, for $ x \in E_{0} $,
\begin{align*} 
\dfrac{ \sqrt{\Delta_n(x)} }{A_n(x)} =  \frac{m\l \big[1+ \O (n^{-1+4a})\big]}{2} \
\sqrt{1+  \dfrac{ r (\l-r) \sin^2(\l x/2) }{  [(\l-r) \sin^2(m\l x/2) + r \sin^2((m+1)\l  x/2)]^2}}, \asntoinfty.
\end{align*}
Thus, by the Kac-Rice formula \eqref{Kac} we obtain
\begin{align*} 
\E[N_n (E_0)] & =
\frac{m\l \big[1+ \O (n^{-1+4a})\big]}{2\pi} \displaystyle \int_{E_0 }
\sqrt{1+  \dfrac{ r (\l-r) \sin^2(\l x/2) }{  [(\l-r) \sin^2(m\l x/2) + r \sin^2((m+1)\l  x/2)]^2}} \, dx.
\end{align*}  
Furthermore, it follows from Corollary 3.2 of \cite{Pir2020} that $ \E[N_n (F_0)]  = \O (m^{1-a}) = \O (n^{1-a})$. 
Hence
\begin{align} \label{P3c19}
& \E[N_n ([0,2\pi/\l])] = \E[N_n (E_0 \cup F_0 )] = \E[N_n (E_0 )]+\E[N_n (F_0 )] \notag \\
& =\frac{m\l \big[1+ \O (n^{-1+4a})\big]}{2\pi} \displaystyle \int_{E_0}
\sqrt{1+  \dfrac{ r (\l-r) \sin^2(\l x/2) }{  [(\l-r) \sin^2(m\l x/2) + r \sin^2((m+1)\l  x/2)]^2}} \, dx + \O (n^{1-a}) \notag \\
& = \frac{m \big[1+ \O (n^{-1+4a})\big]}{\pi} \displaystyle \int_{m^{-a}}^{\pi-m^{-a}}
\sqrt{1+  \dfrac{ r (\l-r) \sin^2(x) }{  [(\l-r) \sin^2(mx) + r \sin^2(m+1)x]^2}} \, dx + \O (n^{1-a}) \notag \\
& = \frac{m + \O (n^{4a})}{\pi} \displaystyle \int_{m^{-a}}^{\pi-m^{-a}}
\sqrt{1+  \dfrac{ r (\l-r) \sin^2(x) }{  [(\l-r) \sin^2(mx) + r \sin^2(m+1)x]^2}} \, dx + \O (n^{1-a}).
\end{align}
Moreover, since $ r, \l-r \geqslant 1 $, we have 
\begin{align*} 
0 \leqslant \dfrac{ \sin(x) }{(\l-r) \sin^2(mx) + r \sin^2(m+1)x} \leqslant \dfrac{ \sin(x) }{ \sin^2(mx) +  \sin^2(m+1)x} = \dfrac{ \sin(x) }{ 1-\cos(x)  \cos(2m+1)x}.
\end{align*} 
Repeating the argument in the proof of Lemma 3.6 of \cite{Pir2021}, one can show that 
\begin{align*} 
\displaystyle \int_{\tfrac{\pi}{2(2m+1)}}^{m^{-a}} \dfrac{ \sin(x) \, dx }{ 1-\cos(x)  \cos(2m+1)x} =\O (m^{-a}),
\end{align*} 	
and similarly
\begin{align*} 
\displaystyle \int_{\pi-m^{-a}}^{\pi-\tfrac{\pi}{2(2m+1)}} \dfrac{ \sin(x) \, dx }{ 1-\cos(x)  \cos(2m+1)x} =\O (m^{-a}).
\end{align*} 
Moreover,
\begin{align*} 
\displaystyle \int_{\tfrac{\pi}{2(2m+1)}}^{\pi-\tfrac{\pi}{2(2m+1)}} \dfrac{ \sin(x) \, dx }{ 1-\cos(x)  \cos(2m+1)x} =\O (1).
\end{align*} 
Note that the last three estimates remain true if we replace $ \tfrac{\pi}{2(2m+1)} $
by $ \tfrac{\pi}{m} $. Taking into account all these facts,  we obtain
\begin{align} \label{P3c20}
& \int_{m^{-a}}^{\pi-m^{-a}}
\sqrt{1+  \dfrac{ r (\l-r) \sin^2(x) }{ [(\l-r) \sin^2(mx) + r \sin^2(m+1)x]^2}} \, dx \notag \\
& = \int_{\tfrac{\pi}{m}}^{\pi-\tfrac{\pi}{m}}
\sqrt{1+  \dfrac{ r (\l-r) \sin^2(x) }{  [(\l-r) \sin^2(mx) + r \sin^2(m+1)x]^2}} \, dx + \O (m^{-a}) =\O (1).
\end{align} 
It follows from \eqref{P3c19} and \eqref{P3c20} that
\begin{align*} 
& \E[N_n ([0,2\pi/\l])] 
= \frac{m}{\pi} \displaystyle \int_{\tfrac{\pi}{m}}^{\pi-\tfrac{\pi}{m}}
\sqrt{1+  \dfrac{ r (\l-r) \sin^2(x) }{  [(\l-r) \sin^2(mx) + r \sin^2(m+1)x]^2}} \, dx + \O (n^{4a})+ \O (n^{1-a}).
\end{align*}
Since $ a \in (0,1/4) $, it is immediate that the best estimate happens at $ a=1/5 $. 
Thus,
\begin{align*} 
& \E[N_n ([0,2\pi/\l])] 
= \frac{m}{\pi} \displaystyle \int_{\tfrac{\pi}{m}}^{\pi-\tfrac{\pi}{m}}
\sqrt{1+  \dfrac{ r (\l-r) \sin^2(x) }{  [(\l-r) \sin^2(mx) + r \sin^2(m+1)x]^2}} \, dx + \O (n^{4/5}), \asntoinfty. \qedhere
\end{align*}
\end{proof}

\belowpdfbookmark{Lemma 3.3}{P2Lem3.3}
\begin{lemma} \label{P2Lem3.3}
Let $ m \in \N $ and $ 1 \leqslant r \leqslant \l-1 $, where $ \l \in \N \setminus\{1\} $ is fixed. We define
\begin{align*}
\I_{\l,r}(m):=\frac{1}{\pi} \displaystyle \int_{\tfrac{\pi}{m}}^{\pi-\tfrac{\pi}{m}}
\sqrt{1+  \dfrac{ r (\l-r) \sin^2(x) }{  [(\l-r) \sin^2(mx) + r \sin^2(m+1)x]^2}} \, dx.
\end{align*}
Then, for sufficiently large $ m $, we have $ \I_{\l,r}(m) = \cC_{\l,r} + \O (m^{-1/5}) $.
\end{lemma}\begin{proof}
Let us define
\begin{align*}
g_{m}(x):= \sqrt{1+  \dfrac{ r (\l-r) \sin^2(x) }{  [(\l-r) \sin^2(mx) + r \sin^2(m+1)x]^2}},
\end{align*}
and $ m':= [m/2] $.
We observe that	
\begin{align*}  
\I_{\l,r}(m)
& = \frac{1}{\pi} \int_{\tfrac{\pi}{m} }^{\pi - \tfrac{\pi}{m} } g_{m}(x) \, dx 
= \frac{2}{\pi} \int_{\tfrac{\pi}{m} }^{\tfrac{\pi}{2}} g_{m}(x) \, dx 
= \frac{2}{\pi} \int_{\tfrac{\pi}{m} }^{\tfrac{m'\pi}{m}} g_{m}(x) \, dx 
+  \frac{2}{\pi} \int_{\tfrac{m'\pi}{m} }^{\tfrac{\pi}{2}} g_{m}(x) \, dx    \\
& = \frac{2}{m \pi} \sum_{k=2}^{m'} \displaystyle \int_{ (k-1) \pi }^{k \pi  }  \sqrt{1+  \dfrac{ r (\l-r) \sin^2(t/m) }{  [(\l-r) \sin^2(t) + r \sin^2(t+t/m)]^2}} \, dt + \O (m^{-1}).
\end{align*}	
For $ 1 \leqslant k \leqslant m' $, let 
\begin{align*} 
\zk 
& := \begin{cases}
k \pi  ,  & \text{if $k$ is odd,} \\
(k-1) \pi , & \text{if $k$ is even.}
\end{cases}
\end{align*}
Then, by the technique established in Lemma 3.9 of \cite{Pir2021}, we are able to show that
\begin{align*}  
& \frac{2}{m \pi} \sum_{k=2}^{m'} \displaystyle \int_{ (k-1) \pi }^{k \pi  }  \sqrt{1+  \dfrac{ r (\l-r) \sin^2(t/m) }{  [(\l-r) \sin^2(t) + r \sin^2(t+t/m)]^2}} \, dt \\
& =  \frac{2}{m \pi} \sum_{k=2}^{m'} \displaystyle \int_{ (k-1) \pi }^{k \pi  }  \sqrt{1+  \dfrac{ r (\l-r) \sin^2(\zk/m) }{  [(\l-r) \sin^2(t) + r \sin^2(t+\zk/m)]^2}} \, dt +\O (m^{-1/5}) .
\end{align*}
It is also easy to show that
\begin{align*}  
& \frac{2}{m \pi} \displaystyle \int_{0}^{ \pi  }  \sqrt{1+  \dfrac{ r (\l-r) \sin^2(\zeta_1/m) }{  [(\l-r) \sin^2(t) + r \sin^2(t+\zeta_1/m)]^2}} \, dt = \O (m^{-1}) .
\end{align*}
Thus, putting all these estimates together, we have
\begin{align}  \label{P3c21}
\I_{\l,r}(m)
& =  \frac{2}{m \pi} \sum_{k=1}^{m'} \displaystyle \int_{ (k-1) \pi }^{k \pi  }  \sqrt{1+  \dfrac{ r (\l-r) \sin^2(\zk/m) }{  [(\l-r) \sin^2(t) + r \sin^2(t+\zk/m)]^2}} \, dt + \O (m^{-1/5}) \notag \\
& =  \frac{2}{\pi^2}  \int_{ 0 }^{ \pi  } \frac{\pi}{m} \sum_{k=1}^{m'} \sqrt{1+  \dfrac{ r (\l-r) \sin^2(\zk/m) }{  [(\l-r) \sin^2(t) + r \sin^2(t+\zk/m)]^2}} \, dt + \O (m^{-1/5}).
\end{align}
Note that the well-known result of Chui \cite{Chu1971} on the rate of convergence of Riemann sums guarantees that (Lemma 3.11 of \cite{Pir2021} discusses such a rate in detail)
\begin{align}  \label{P3c22}
& \int_{ 0 }^{ \pi  } \frac{\pi}{m} \sum_{k=1}^{m'} \sqrt{1+  \dfrac{ r (\l-r) \sin^2(\zk/m) }{  [(\l-r) \sin^2(t) + r \sin^2(t+\zk/m)]^2}} \, dt \notag \\
& = \int_{ 0 }^{ \pi  } \int_{ 0 }^{ \pi/2 } \sqrt{1+  \dfrac{ r (\l-r) \sin^2(s) }{  [(\l-r) \sin^2(t)+  r \sin^2(s+t)] ^2 }} \, ds 
\, dt 
+  \O   \bigg( \frac{\log n}{n} \bigg).
\end{align}
Now, \eqref{P3c21} and \eqref{P3c22} yield
\begin{align*}  
\I_{\l,r}(m)
& =  \frac{2}{\pi^2}  \int_{ 0 }^{ \pi  } \int_{ 0 }^{ \pi/2 } \sqrt{1+  \dfrac{ r (\l-r) \sin^2(s) }{  [(\l-r) \sin^2(t)+  r \sin^2(s+t)] ^2 }} \, ds  \, dt + \O (m^{-1/5}) \\
& =  \frac{1}{\pi^2}  \int_{ 0 }^{ \pi  } \int_{ 0 }^{ \pi } \sqrt{1+  \dfrac{ r (\l-r) \sin^2(s) }{  [(\l-r) \sin^2(t)+  r \sin^2(s+t)] ^2 }} \, ds  \, dt + \O (m^{-1/5}) \\
& = \cC_{\l,r} + \O (m^{-1/5}), \asmtoinfty,
\end{align*}
as desired.
\end{proof}
Finally, we get to the point to prove Theorem \ref{PeriodicTn-nonzero}.

\belowpdfbookmark{Proof of Theorem 2.3}{PeriodicTn-nonzero}
\begin{proof}[Proof of Theorem \ref{PeriodicTn-nonzero}]
As a direct consequence of Lemmas \ref{P2Lem3.2} and \ref{P2Lem3.3}, we have
\begin{align*} 
& \E[N_n ([0,2\pi/\l])] 
= m \I_{\l,r}(m) + \O (n^{4/5}) = m \cC_{\l,r} + \O (n^{4/5}), \asntoinfty.
\end{align*}
Similarly, we can show that $ \E[N_n ([2k\pi/\l,2(k+1)\pi/\l ])] = m \cC_{\l,r} + \O (n^{4/5}) $, $ 1 \leqslant k \leqslant \l-1 $, as $ n $ grows to infinity. Therefore,
\begin{align*}
\E [ N_n (0,2\pi) ] 
& = \sum_{k=0}^{\l -1} \E [ N_n ([2k\pi/\l,2(k+1)\pi/\l ]) ] = m \l \, \cC_{\l,r} + \O (n^{4/5}) = n \, \cC_{\l,r} + \O (n^{4/5}), \asntoinfty. \qedhere
\end{align*}
\end{proof}
Before proving Lemma \ref{boundedcKrl}, we need a lemma.
\begin{lemma}
For $ \alpha \in (0, \pi/2) $, we have
\begin{equation*}
I_{\alpha}:= \int_{ 0 }^{ \pi } \int_{ 0 }^{ \pi }
\dfrac{ \sin(s) \, ds \, dt }{  \sin^2(\alpha) \sin^2(t)+  \cos^2(\alpha) \sin^2(s+t) }  = \dfrac{\pi^2}{\sin(\alpha)\cos(\alpha)}.
\end{equation*}
\end{lemma}
\begin{proof}
Set $ v=t $ and $ u=s+t $. We use these change of variables to rewrite $ I_{\alpha} $ as
\begin{align*}
I_{\alpha} & = \int_{ 0 }^{ \pi } \int_{ v }^{ \pi+v }
\dfrac{ \sin(u-v) \, du \, dv }{  \sin^2(\alpha) \sin^2(v)+  \cos^2(\alpha) \sin^2(u) }  \\
& = \int_{ 0 }^{ \pi } \int_{ v }^{ \pi+v }
\dfrac{ \sin(u)\cos(v) \, du \, dv }{  \sin^2(\alpha) \sin^2(v)+  \cos^2(\alpha) \sin^2(u) } - \int_{ 0 }^{ \pi } \int_{ v }^{ \pi+v }
\dfrac{ \sin(v)\cos(u) \, du \, dv }{  \sin^2(\alpha) \sin^2(v)+  \cos^2(\alpha) \sin^2(u) } \\
& = : I_{\alpha,1} -I_{\alpha,2} .
\end{align*}
For $ a,b>0 $, employing change of variables $ z=\sin(y) $ shows that
\begin{align} \label{P3c23}
\dis \int \dfrac{\cos(y) \, dy}{a^2\sin^2(y)+b^2 } & = \dfrac{1}{ab} \arctan \bigg( \dfrac{a\sin(y)}{b}\bigg)+C.
\end{align}
Note that
\begin{align*}
I_{\alpha,1} & := \int_{ 0 }^{ \pi } \int_{ v }^{ \pi+v }
\dfrac{ \sin(u)\cos(v) \, du \, dv }{  \sin^2(\alpha) \sin^2(v)+  \cos^2(\alpha) \sin^2(u) } \\
& = \int_{ 0 }^{ \pi } \int_{ 0 }^{ u }
\dfrac{ \sin(u)\cos(v) \, dv \, du }{  \sin^2(\alpha) \sin^2(v)+  \cos^2(\alpha) \sin^2(u) } + \int_{ \pi }^{ 2\pi } \int_{ u-\pi }^{ \pi }
\dfrac{ \sin(u)\cos(v) \, dv \, du }{  \sin^2(\alpha) \sin^2(v)+  \cos^2(\alpha) \sin^2(u) }.
\end{align*}
We apply \eqref{P3c23}, while setting $ a=\sin(\alpha) $, $ b =  \cos(\alpha) \sin(u)$, and observe that
\begin{align*}
\int_{ 0 }^{ \pi } \int_{ 0 }^{ u }
\dfrac{ \sin(u)\cos(v) \, dv \, du }{  \sin^2(\alpha) \sin^2(v)+  \cos^2(\alpha) \sin^2(u) } 
& = \int_{ 0 }^{ \pi } \dfrac{ \sin(u) }{\sin(\alpha)\cos(\alpha) \sin(u)} \arctan \bigg( \dfrac{ \sin(\alpha)\sin(v)}{\cos(\alpha) \sin(u)}\bigg) \bigg|_{v=0}^u \, du \\
& = \int_{ 0 }^{ \pi } \dfrac{ \alpha \, du }{\sin(\alpha)\cos(\alpha)} = \dfrac{ \pi \alpha }{\sin(\alpha)\cos(\alpha)}.
\end{align*}
Similarly, we can show that
\begin{align*}
\int_{ \pi }^{ 2\pi } \int_{ u-\pi }^{ \pi }
\dfrac{ \sin(u)\cos(v) \, dv \, du }{  \sin^2(\alpha) \sin^2(v)+  \cos^2(\alpha) \sin^2(u) } = \dfrac{ \pi \alpha }{\sin(\alpha)\cos(\alpha)},
\end{align*}
which implies that 
\begin{align*}
I_{\alpha,1}  = \dfrac{2 \pi \alpha }{\sin(\alpha)\cos(\alpha)},
\end{align*}
Note that with help of \eqref{P3c23}, while setting $ a=\cos(\alpha) $, $ b =  \sin(\alpha) \sin(v)$, we have
\begin{align*}
I_{\alpha,2} & := \int_{ 0 }^{ \pi } \int_{ v }^{ \pi+v }
\dfrac{ \sin(v)\cos(u) \, du \, dv }{  \sin^2(\alpha) \sin^2(v)+  \cos^2(\alpha) \sin^2(u) } \\
& = \int_{ 0 }^{ \pi } \dfrac{ \sin(v) }{\sin(\alpha)\cos(\alpha) \sin(v)} \arctan \bigg( \dfrac{ \cos(\alpha)\sin(u)}{\sin(\alpha)\sin(v)}\bigg) \bigg|_{u=v}^{\pi+v} \, dv \\
& = -2 \int_{ 0 }^{ \pi } \dfrac{ \arctan(\cot(\alpha)) \, dv}{\sin(\alpha)\cos(\alpha) } = \dfrac{ 2\pi\alpha -\pi^2}{\sin(\alpha)\cos(\alpha) }.
\end{align*}
Therefore, 
\begin{align*}
I_{\alpha} & = I_{\alpha,1} -I_{\alpha,2} = \dfrac{ \pi^2}{\sin(\alpha)\cos(\alpha) }. \qedhere
\end{align*}
\end{proof}

\currentpdfbookmark{Proof of Lemma 2.2}{P2Lem2.2}
\begin{proof}[Proof of Lemma \ref{boundedcKrl}] \label{P2Lem2.2}
It is immediate that there exists $ \alpha_0 \in (0,\pi/2) $ such that $ \sin(\alpha_0) = \sqrt{1-r/\l} $ and $ \cos(\alpha_0) = \sqrt{r/\l} $. Therefore, for $ 0<r<\l $, by the previous lemma we have
\begin{align*}
J_{\l,r}:&=\frac{1}{\pi ^2} \displaystyle \int_{ 0 }^{ \pi } \int_{ 0 }^{ \pi }
\dfrac{ \sqrt{r (\l-r)} \sin(s) \, ds \, dt }{ (\l-r) \sin^2(t)+  r \sin^2(s+t)} \\
&= \frac{1}{\pi ^2} \displaystyle \int_{ 0 }^{ \pi } \int_{ 0 }^{ \pi }
\dfrac{ \sin(\alpha_0) \cos(\alpha_0)\sin(s) \, ds \, dt }{ \sin^2(\alpha_0) \sin^2(t)+  \cos^2(\alpha_0) \sin^2(s+t) }=1.
\end{align*}
The upper bound for $ \cC_{\l,r} $ is immediate since $ \cC_{\l,r} \leqslant 1+ J_{\l,r} =2  $.
The lower bound is also trivial by Jensen's inequality.
Define $ \Phi(y) := \sqrt{1+y^2} $, and note that $ \Phi $ is a strictly convex function on $ [0,\infty) ,$ which implies that for any nonconstant $ f \in \mathbb{L}^1  ((0,\pi))$, Jensen's inequality is strict (see Theorem 3.3 of \cite{Rud1974}), i.e.,
\begin{align*} 
& \dfrac{1}{\pi} \int_{ 0 }^{ \pi } \Phi (f(s)) \, ds > \Phi \bigg(\displaystyle \dfrac{1}{\pi} \int_{ 0 }^{ \pi } f(s) \, ds \bigg).
\end{align*} 
Let us define
\begin{equation*}
	f_{\l,r}(s,t):=\frac{ \sqrt{r (\l-r)} \sin(s)}{ (\l-r) \sin^2(t)+  r \sin^2(s+t)}, \quad (s,t) \in (0,\pi)\times (0,\pi).
\end{equation*}
Hence by Jensen's inequality we have
\begin{align*} 
\cC_{\l,r} 
& = \frac{1}{\pi ^2} \displaystyle \int_{ 0 }^{ \pi } \int_{ 0 }^{ \pi }
\sqrt{1+  f_{\l,r}^2(s,t)} \, ds \, dt \notag 
> \frac{1}{\pi } \int_{ 0 }^{ \pi } 
\sqrt{1+  \bigg[\frac{1}{\pi } \int_{ 0 }^{ \pi } f_{\l,r}(s,t) \, ds\bigg]^2}  dt \\
& \geqslant  \displaystyle  
\sqrt{1+  \bigg[\frac{1}{\pi^2 } \int_{ 0 }^{ \pi } \int_{ 0 }^{ \pi } f_{\l,r}(s,t) \, ds \, dt\bigg]^2} 
= \sqrt{1+ J_{\l,r}^2} =\sqrt{2},
\end{align*} 
which concludes the proof.
\end{proof}

\section*{Acknowledgment}
I am very grateful to Igor Pritsker for having introduced me to this project. The first part of this manuscript was one of the main contributions to the author's Ph.D. dissertation \cite{Pir-diss} under his supervision. The author would also like to thank the OSU Foundation for their financial support through the ``2020 Distinguished Graduate Fellowship", and the OSU Graduate College through the ``2020 Robberson Summer Dissertation Fellowship".

\bibliographystyle{amsplain}
\bibliography{P3references}

\end{document}